\documentclass{article}
\usepackage{amsmath}
\usepackage{amssymb}

\newtheorem{theorem}{Theorem}

\newtheorem{corollary}[theorem]{Corollary}

\newtheorem{definition}[theorem]{Definition}
\newtheorem{example}[theorem]{Example}

\newtheorem{proposition}[theorem]{Proposition}
\newtheorem{remark}[theorem]{Remark}

\newenvironment{proof}[1][Proof]{\noindent\textbf{#1.} }{\ \rule{0.5em}{0.5em}}
\input{tcilatex}
\begin{document}

\title{A Topological Approach to Soft Covering Approximation Space }
\author{Naime Tozlu\thanks{%
naimetozlu@nigde.edu.tr}, Saziye Yuksel\thanks{%
syuksel@selcuk.edu.tr} and Tugba Han Simsekler\thanks{%
tsimsekler@hotmail.com.tr} \\
{\small Nigde University, Science and Art Faculty, Department of
Mathematics, Nigde, Turkey. }\\
{\small Selcuk University, Science Faculty, Department of Mathematics,
Konya, Turkey. }\\
{\small Kafkas University}, {\small Science and Art Faculty, Department of
Mathematics, Kars, Turkey.}}
\date{}
\maketitle

\begin{abstract}
Theories of rough sets and soft sets are powerful mathematical tools for
modelling various types of vagueness. Hybrid model combining a rough set
with a soft set which is called soft rough set proposed by Feng et al. [3]
in 2010. In this paper, we study soft covering based rough sets from the
topological view. We present under which conditions soft covering lower
approximation operation become interior operator and the soft covering upper
approximation become closure operator. Also some new methods for generating
topologies are obtained. Finally, we study the relationship between concepts
of topology and soft covering lower and soft covering upper approximations.%
\newline

\noindent \textbf{Keywords :} Soft set, rough set, soft covering based rough
set, topology

2010 Mathematics Subject Classification: 54A0J, 68T30, 68T37, 03B52
\end{abstract}

\section{\protect\bigskip Introduction}

Mathematics is based on exact concepts and there is not vagueness for
mathematical concepts. For this reason researchers need to define some new
concepts for vagueness. The most successful approach is exactly Zadeh's
fuzzy set \cite{n} which is based on membership function. This theorical
approach is used in several areas as engineering, medicine, economics and
etc. Pawlak \cite{5} initiated rough set theory\ in 1982 as a tool for
uncertainty and imprecise data. The theory is based on partition or
equivalence relation, which is rather strict. Covering based rough set \cite%
{2,m} is an important extension of rough sets. Compared with rough sets, it
often gives a more reasonable description to a subset of the universe. In
recent years, covering based rough set theory has attracted more attentions.
The studies of Zhu et al. \cite{10,11,12,13} are fundamental and
significant. In 1999 Molodtsov \cite{4} gave soft set theory as a new tool
for vagueness and showed in his paper that soft set theory can be applied to
several areas. The hybrid models like fuzzy soft set \cite{8}, rough soft
set \cite{3}, soft rough set \cite{3} took attention from researchers. Feng
et al. investigated the concept of soft rough set \cite{3} which is a
combination of soft set and rough set. It is known that the equivalence
relation is used to form the granulation structure of the universe in the
rough set model and also the soft set is used to form the granulation
structure of the universe in the soft rough set model.

Topology is a branch of mathematics, whose concepts exist not only in almost
all branches of mathematics, but also in many real life applications.
Topology is also a mathematical tool to study rough sets \cite{c,d,e,f,g}.
It should be noted that the generation of topology by relation and the
representation of topological concepts via relation will narrow the gap
between topology and its applications.

The remaining part of this paper is arranged as follows:

In section 3, we give a new concept called as soft covering based rough sets
and its basic properties. Also we investigate the conditions under which the
soft covering lower and upper approximation operations are also interior and
closure operators, respectively. In section 4, we discuss methods of setting
up topology in soft covering approximation space. The relationship between
concepts of topology and soft covering lower and upper approximations are
studied in section 5. and the special condition of soft covering
approximation space is investigated in section 6.

\section{Preliminaries}

In this section, we introduce the fundamental ideas behind rough sets, soft
sets and topological spaces.

First, we recall some concepts and properties of the Pawlak's rough sets.

\begin{definition}
\cite{5} Let $U$ be a finite set and $R$ be an equivalence relation on $U$.
Then the pair $(U,R)$ is called a Pawlak approximation space. $R$ generates
a partition $U/R=\{Y_{1},Y_{2},...,Y_{m}\}$ on $U$ \bigskip where $%
Y_{1},Y_{2},...,Y_{m}$ are the equivalence classes generated by the
equivalence relation $R$. In the rough set theory, these are also called
elementary sets of $R$. For any $X\subseteq U$, we can describe $X$ by the
elementary sets of $R$ and the two sets:%
\begin{align*}
R_{-}(X) & =\cup\{Y_{i}\in U/R:Y_{i}\subseteq X\}, \\
R^{-}(X) & =\cup\{Y_{i}\in U/R:Y_{i}\cap X\neq\emptyset\}
\end{align*}
which are called the lower and the upper approximation of $X$, respectively.
In addition,%
\begin{align*}
POS_{R}(X) & =R_{-}(X), \\
NEG_{R}(X) & =U-R^{-}(X), \\
BND_{R}(X) & =R^{-}(X)-R_{-}(X)
\end{align*}
are called the positive, negative and boundary regions of $X$, respectively.
\end{definition}

Now, we are ready to give the definition of rough sets:

\begin{definition}
\cite{7} Let $(U,R)$ be a Pawlak approximation space. A subset $X\subseteq U$
is called definable (crisp) if $R_{-}(X)=R^{-}(X)$; in the opposite case,
i.e., if $BND_{R}(X)\neq\emptyset$, $X$ is said to be rough(or inexact). Any
pair of the form $R(X)=(R_{-}(X),R^{-}(X))$ is called a rough set of $X$.
\end{definition}

Let $U$ be an initial universe set and $E$ be the set of all possible
parameters with respect to $U$. Usually, parameters are attributes,
characteristics or properties of the objects in $U$. The notion of a soft
set is defined as follows:

\begin{definition}
\cite{4} A pair $G=(F,A)$ is called a soft set over $U$, where $A\subseteq E$
and $F:A\longrightarrow P(U)$ is a set-valued mapping.
\end{definition}

\begin{theorem}
\cite{1} Every rough set may be considered as a soft set.
\end{theorem}

\ \ The following result indicates that soft sets and binary relations are
closely related.

\begin{theorem}
\cite{3} Let $G=(F,A)$ be a soft set over $U$. Then $G$ induces a binary
relation $R_{G}\subseteq A\times U$, which is defined by%
\begin{equation*}
(x,y)\in R_{G}\Longleftrightarrow y\in F(x)
\end{equation*}
where $x\in A$, $y\in U$. Conversely, assume that $R$ is a binary relation
from $A$ to $U$. Define a set valued mapping $F_{R}:A\longrightarrow P(U)$ by%
\begin{equation*}
F_{R}(x)=\{y\in U:(x,y)\in R\}\text{,}
\end{equation*}
where $x\in A$. Then $G_{R}=(F_{R},A)$ is a soft set over $U$. Moreover, it
is seen that $G_{R_{G}}=G$ and $R_{G_{R}}=R$.
\end{theorem}

\begin{definition}
\cite{3} Let $G=(F,A)$ be a soft set over $U$. Then the pair $S=(U,R_{G})$
is called a soft approximation space.
\end{definition}

\section{ Soft Covering Based Rough Sets}

We know that a soft set is determined by the set-valued mapping from a set
of parameters to the powerset of the universe. In this section, we will use
a special soft set and by using this soft set, we will establish a soft
covering approximation space.

\begin{definition}
\cite{3} A soft set $G=(F,A)$ over $U$ is called a full soft set if $%
\dbigcup \limits_{a\in A}F(a)=U$.
\end{definition}

\begin{definition}
\cite{3} A full soft set $G=(F,E)$ over $U$ is called a covering soft set if 
$F(e)\neq \emptyset ,\ \forall e\in E$.
\end{definition}

We discussed some properties of soft covering upper and lower approximations
in our previous work \cite{9}. Following definitions are given in this paper.

\begin{definition}
Let $G=(F,E)$ be a covering soft set over $U$. The ordered pair $S=(U,C_{G})$
is called a soft covering approximation space.
\end{definition}

\begin{definition}
Let $S=(U,C_{G})$ be a soft covering approximation space, for any $x\in U$,
the soft minimal description of $x$ is defined as following:%
\begin{equation*}
Md_{S}(x)=\{F(e):e\in E\wedge x\in F(e)\wedge (\forall a\in E\wedge x\in
F(a)\subseteq F(e)\Longrightarrow F(a)=F(e))\}\text{.}
\end{equation*}
\end{definition}

\begin{definition}
Let $S=(U,C_{G})$ be a soft covering approximation space. For a set $%
X\subseteq U$, the soft covering lower and upper approximations are
respectively defined as%
\begin{align*}
S_{-}(X)& =\cup \{F(e):e\in E\wedge F(e)\subseteq X\} \\
S^{-}(X)& =S_{-}(X)\cup \{Md_{S}(x):x\in X-S_{-}(X)\}\text{.}
\end{align*}%
In addition,%
\begin{align*}
POS_{S}(X)& =S_{-}(X) \\
NEG_{S}(X)& =U-S^{-}(X) \\
BND_{S}(X)& =S^{-}(X)-S_{-}(X)
\end{align*}%
are called the soft covering positive, negative and boundary regions of $X$,
respectively.
\end{definition}

\begin{definition}
Let $S=(U,C_{G})$ be a soft covering approximation space. A subset $%
X\subseteq U$ is called definable if $S_{-}(X)=S^{-}(X)$; in the opposite
case, i.e., if $S_{-}(X)\neq S^{-}(X)$, $X$ is said to be soft covering
based rough set. The pair $(S_{-}(X),S^{-}(X))$ is called soft covering
based rough set of $X$ and it is showed \ that $X=(S_{-}(X),S^{-}(X))$.
\end{definition}

\begin{example}
Let $S=(U,C_{G})$ be a soft covering approximation space, where $%
U=\{a,b,c,d,e,f,g,h\},\ E=\{e_{1},e_{2},e_{3},e_{4},e_{5}\},\
F(e_{1})=\{a,b\},\ F(e_{2})=\{b,c,d\},\ F(e_{3})=\{e,f\},\ F(e_{4})=\{g\}$
and $F(e_{5})=\{g,h\}$. For $X_{1}=\{a,b,c\}\subseteq U$, since $%
S_{-}(X_{1})\neq S^{-}(X_{1})$, $X_{1}$ is a soft covering based rough set.
For $X_{2}=\{e,f,g\}\subseteq U$, since $S_{-}(X_{2})=S^{-}(X_{2})$, $X_{2}$
is a definable set.
\end{example}

We give following two theorems in our previous work \cite{9}.

\begin{theorem}
Let $G=(F,E)$ be a soft set over $U$, $S=(U,C_{G})$ be a soft covering
approximation space and $X,Y\subseteq U$. Then the soft covering lower and
upper approximations have the following properties:
\end{theorem}

\begin{enumerate}
\item $S_{-}(U)=S^{-}(U)=U$

\item $S_{-}(\emptyset )=S^{-}(\emptyset )=\emptyset $

\item $S_{-}(X)\subseteq X\subseteq S^{-}(X)$

\item $X\subseteq Y\Longrightarrow S_{-}(X)\subseteq S_{-}(Y)$

\item $S_{-}(S_{-}(X))=S_{-}(X)$

\item $S^{-}(S^{-}(X))=S^{-}(X)$

\item $\forall e\in E,\ S_{-}(F(e))=F(e)$

\item $\forall e\in E,\ S^{-}(F(e))=F(e)$
\end{enumerate}

\begin{theorem}
Let $G=(F,E)$ be a soft set over $U$, $S=(U,C_{G})$ be a soft covering
approximation space and $X,Y\subseteq U$. Then the soft covering lower and
upper approximations do not have the following properties:
\end{theorem}

\begin{enumerate}
\item $S_{-}(X\cap Y)=S_{-}(X)\cap S_{-}(Y)$

\item $S^{-}(X\cup Y)=S^{-}(X)\cup S^{-}(Y)$

\item $X\subseteq Y\Longrightarrow S^{-}(X)\subseteq S^{-}(Y)$

\item $S_{-}(X)=-(S^{-}(-X))$

\item $S^{-}(X)=-(S_{-}(-X))$

\item $S_{-}(-S_{-}(X))=-S_{-}(X)$

\item $S^{-}(-S^{-}(X))=-S^{-}(X)$
\end{enumerate}

The symbol "-" denotes the complement of the set. The following examples
show that the equalities mentioned above do not hold.

\begin{example}
Let $S=(U,C_{G})$ be a soft covering approximation space, where $%
U=\{a,b,c,d,e,f,g\},\ E=\{e_{1},e_{2},e_{3},e_{4}\},\ F(e_{1})=\{a,b,c\},\
F(e_{2})=\{b,c,d\},\ F(e_{3})=\{d,e\}$ and $F(e_{4})=\{f,g\}$. Suppose that $%
X=\{a,b,c,d\}\subseteq U$ and $Y=\{d,e\}.$ The properties 1, 4, 5, 6, 7 of
Theorem 15 do not hold.
\end{example}

\begin{example}
Let $S=(U,C_{G})$ be a soft covering approximation space and $(F,E)$ be a
soft set given in the Example 16. Suppose that $X=\{a,b\}\subseteq U$ and $%
Y=\{c,d\}\subseteq U$. The property 2 of Theorem 15 does not hold.
\end{example}

\begin{example}
Let $S=(U,C_{G})$ be a soft covering approximation space and $(F,E)$ be a
soft set given in the Example 16. Suppose that $X=\{d\}\subseteq U$ and $%
Y=\{b,c,d\}\subseteq U$. The property 3 of Theorem 15 does not hold.
\end{example}

Now, we consider under which conditions soft covering lower and upper
approximations satisfy properties 1, 2, 3 of Theorem 15.

The continuation of the paper, the parameter set $E$ is supposed to be
finite.

\begin{proposition}
$S_{-}(X)=X$ if and only if $X$ is a union of some elements of $C_{G}$.
Similarly, $S^{-}(X)=X$ if and only if $X$ is a union of some elements of $%
C_{G}$.
\end{proposition}

\begin{theorem}
Let $S=(U,C_{G})$ be a soft covering approximation space and $X,Y\subseteq U$%
. $S_{-}(X\cap Y)=S_{-}(X)\cap S_{-}(Y)$ if and only if $\forall
e_{1},e_{2}\in E,\ F(e_{1})\cap F(e_{2})$ is a finite union of elements of $%
C_{G}$.
\end{theorem}

\begin{proof}
$\Longrightarrow:$ Since $F(e_{1})\cap F(e_{2})=S_{-}(F(e_{1}))\cap
S_{-}(F(e_{2}))=S_{-}(F(e_{1})\cap F(e_{2}))$ and $S_{-}(F(e_{1})\cap
F(e_{2}))$ is a finite union of elements of $C_{G}$, $F(e_{1})\cap F(e_{2})$
is a finite union of elements of $C_{G}$.

$\Longleftarrow :$ By 4 of Theorem 14, it is easy to see that $S_{-}(X\cap
Y)\subseteq S_{-}(X)\cap S_{-}(Y)$. Now we shall show that $S_{-}(X)\cap
S_{-}(Y)\subseteq S_{-}(X\cap Y)$. Let $S_{-}(X)=F(e_{1})\cup F(e_{2})\cup
...\cup F(e_{m})$ and $S_{-}(Y)=F(e_{1}^{^{\prime }})\cup F(e_{2}^{^{\prime
}})\cup ...\cup F(E_{n}^{^{\prime }})$ where $e_{i},\ e_{j}^{\prime }\in E,\
1\leq i\leq m,\ 1\leq j\leq n$. For any $1\leq i\leq m$ and $1\leq j\leq n,\
F(e_{i})\cap F(e_{j}^{^{\prime }})\subseteq X\cap Y$ and $F(e_{i})\cap
F(e_{j}^{^{\prime }})$ is a finite union of elements of $C_{G}$, let us say $%
F(e_{i})\cap F(e_{j}^{^{\prime }})=F(p_{1})\cup ...\cup F(p_{l})$ where $%
F(p_{h})\in C_{G},\ 1\leq h\leq l$, so $F(p_{h})\subseteq S_{-}(X\cap Y)$
for $1\leq h\leq l$. Thus $F(e_{i})\cap F(e_{j}^{^{\prime }})\subseteq
S_{-}(X\cap Y)$ for $1\leq i\leq m$ and $1\leq j\leq n$. From $S_{-}(X)\cap
S_{-}(Y)=\dbigcup\limits_{i=1}^{m}\dbigcup\limits_{j=1}^{n}\left[
F(e_{i})\cap F(e_{j}^{^{\prime }})\right] $, hence $S_{-}(X)\cap
S_{-}(Y)\subseteq S_{-}(X\cap Y)$.
\end{proof}

\begin{theorem}
Let $S=(U,C_{G})$ be a soft covering approximation space and $X,Y\subseteq U$%
. $X\subseteq Y\Longrightarrow S^{-}(X)\subseteq S^{-}(Y)$ if and only if $%
\forall e_{1},e_{2}\in E,\ F(e_{1})\cap F(e_{2})$ is a finite union of
elements of $C_{G}$.
\end{theorem}

\begin{proof}
$\Longrightarrow :$ $S^{-}(F(e_{1})\cap F(e_{2}))\subseteq
S^{-}(F(e_{1}))=F(e_{1})$ and $S^{-}(F(e_{1})\cap F(e_{2}))\subseteq
S^{-}(F(e_{2}))=F(e_{2})$, so $S^{-}(F(e_{1})\cap F(e_{2}))\subseteq
F(e_{1})\cap F(e_{2})$. By property 3 of Theorem 14, $F(e_{1})\cap
F(e_{2})\subseteq S^{-}(F(e_{1})\cap F(e_{2}))$, so $F(e_{1})\cap
F(e_{2})=S^{-}(F(e_{1})\cap F(e_{2}))$. Hence, $F(e_{1})\cap F(e_{2})$ is a
finite union of elements of $C_{G}$.

$\Longleftarrow :$ By the definition of soft covering upper approximation, $%
S^{-}(X)$ can be expressed as $S^{-}(X)=S_{-}(X)\cup F(e_{1})\cup ...\cup
F(e_{m})$ where $y_{i}\in F(e_{i})\nsubseteq X$ and $F(e_{i})\in
Md_{S}(y_{i})$ for some $y_{i}\in X-S_{-}(X),\ 1\leq i\leq m$. It is obvious
that $y_{i}\in Y$. If $y_{i}\in Y-S_{-}(Y)$, it is easy to see that $%
F(e_{i})\subseteq S^{-}(Y),\ 1\leq i\leq m$. If $y_{i}\notin Y-S_{-}(Y)$,
then $y_{i}\in S_{-}(Y)$. Thus, there exists a $F(e_{j})\in C_{G}$ such that 
$y_{i}\in F(e_{j})\subseteq S_{-}(Y)$. By the assumption of this Theorem, $%
F(e_{i})\cap F(e_{j})$ is a finite union of elements in $C_{G}$. Let us say $%
F(e_{i})\cap F(e_{j})=F(e_{1})\cup ...\cup F(e_{l})$ where $F(e_{h})\in
C_{G},\ 1\leq h\leq l$, so there exists $1\leq j\leq l$ such that $y_{i}\in
F(e_{j})$. Since $F(e_{i})\in Md_{S}(y_{i}),\ F(e_{i})=F(e_{j})$, thus $%
F(e_{i})\subseteq F(e_{j})$. Therefore, $F(e_{i})\subseteq S_{-}(Y)\subseteq
S^{-}(Y),\ 1\leq i\leq m$. From property 3 and property 4 of Theorem 14, $%
S_{-}(X)\subseteq S_{-}(Y)\subseteq S^{-}(Y)$, so $S^{-}(X)\subseteq S^{-}(Y)
$.
\end{proof}

\begin{theorem}
Let $S=(U,C_{G})$ be a soft covering approximation space and $X,Y\subseteq U$%
. $X\subseteq Y\Longrightarrow S^{-}(X)\subseteq S^{-}(Y)$ if and only if $%
S^{-}(X\cup Y)=S^{-}(X)\cup S^{-}(Y)$.
\end{theorem}

\begin{proof}
$\Longrightarrow :$ By the assumption of this Theorem, $S^{-}(X)\subseteq
S^{-}(X\cup Y)$ and $S^{-}(Y)\subseteq S^{-}(X\cup Y)$, so $S^{-}(X)\cup
S^{-}(Y)\subseteq S^{-}(X\cup Y)$. Now we shall show that $S^{-}(X\cup
Y)\subseteq S^{-}(X)\cup S^{-}(Y)$. By property 3 of Theorem 14, $X\cup
Y\subseteq S^{-}(X)\cup S^{-}(Y)$. By the assumption of this Theorem, $%
S^{-}(X\cup Y)\subseteq S^{-}\left( S^{-}(X)\cup S^{-}(Y)\right) $. By
Proposition 19, $S^{-}\left( S^{-}(X)\cup S^{-}(Y)\right) =S^{-}(X)\cup
S^{-}(Y)$, so $S^{-}(X\cup Y)\subseteq S^{-}(X)\cup S^{-}(Y)$.

$\Longleftarrow:$ If $X\subseteq Y,\ S^{-}(Y)=S^{-}(X\cup Y)=S^{-}(X)\cup
S^{-}(Y)$, so $S^{-}(X)\subseteq S^{-}(Y)$.
\end{proof}

\begin{corollary}
Let $S=(U,C_{G})$ be a soft covering approximation space and $X,Y\subseteq U$%
. $S^{-}(X\cup Y)=S^{-}(X)\cup S^{-}(Y)$ if and only if $\forall
e_{1},e_{2}\in E,\ F(e_{1})\cap F(e_{2})$ is a finite union of elements in $%
C_{G}$.
\end{corollary}

\begin{proof}
The proof is obvious by Theorem 21 and Theorem 22.
\end{proof}

\section{Some methods to set up topologies in soft covering approximation
space}

\begin{theorem}
Let $U$ be a nonempty universe set and $S=(U,C_{G})$ be a soft covering
approximation space. For each $e_{1},e_{2}\in E,\ F(e_{1})\cap F(e_{2})$ is
a finite union of elements of $C_{G}$.%
\begin{equation*}
\tau =\{X\subseteq U:S_{-}(X)=X\}
\end{equation*}%
be a collection of subsets of $U$. Then $\tau $ is called a topology over $U$%
.
\end{theorem}

\begin{proof}
\ \ \ \ \ \ \ \ \ \ \ \ \ \ \ \ \ \ \ \ \ \ \ \ \ \ \ \ \ \ \ \ \ \ \ \ \ \
\ \ \ \ \ \ \ \ \ \ \ \ \ \ \ \ \ \ \ \ \ \ \ \ \ \ \ \ \ \ \ \ \ \ \ \ \ \
\ \ \ \ \ \ \ \ \ \ \ \ \ \ \ \ \ \ \ \ \ \ \ \ \ \ \ \ \ \ \ \ \ \ \ \ \ \
\ \ \ \ \ \ \ \ 

\begin{description}
\item[$O_{1})$] If $X=\emptyset $, then by Theorem 14, $S_{-}(\emptyset
)=\emptyset $. Hence $\emptyset \in \tau $. If $X=U$, then by Theorem 14, $%
S_{-}(U)=U$. Hence $U\in \tau $.

\item[$O_{2})$] Let for each $i\in I$, $A_{i}\in \tau $, i.e., $%
S_{-}(A_{i})=A_{i}$. Then there exists a $j\in I,\ A_{j}\in \tau $ such that 
$A_{j}\subseteq \dbigcup\limits_{i\in I}A_{i}$. From Theorem 14, $%
S_{-}(A_{j})\subseteq S_{-}(\dbigcup\limits_{i\in I}A_{i})$. Since $A_{j}\in
\tau ,\ S_{-}(A_{j})=A_{j}$. Hence $A_{j}\subseteq
S_{-}(\dbigcup\limits_{i\in I}A_{i})$. Since this property is satisfied for
each $j\in I$, we get%
\begin{equation}
\dbigcup\limits_{i\in I}A_{i}\subseteq S_{-}(\dbigcup\limits_{i\in I}A_{i})
\label{1}
\end{equation}%
Also by Theorem 14, we know that%
\begin{equation}
S_{-}(\dbigcup\limits_{i\in I}A_{i})\subseteq \dbigcup\limits_{i\in I}A_{i}
\label{2}
\end{equation}%
\allowbreak From (\ref{1}) and (\ref{2}), we get $S_{-}(\dbigcup\limits_{i%
\in I}A_{i})=\dbigcup\limits_{i\in I}A_{i}$. And so we conclude that $%
\dbigcup\limits_{i\in I}A_{i}\in \tau $.

\item[$O_{3})$] Let $A,B\in \tau $. Hence we get $S_{-}(A)=A$ and $S_{-}(B)=B
$. By Theorem 20, $S_{-}(A\cap B)=S_{-}(A)\cap S_{-}(B)=A\cap B$. Hence $%
S_{-}(A\cap B)=A\cap B$. Therefore $A\cap B\in \tau $.
\end{description}
\end{proof}

\begin{theorem}
Let $U$ be a nonempty universe set and $S=(U,C_{G})$ be a soft covering
approximation space$.$For each $e_{1},e_{2}\in E,\ F(e_{1})\cap F(e_{2})$ is
a finite union of elements of $C_{G}$.%
\begin{equation*}
K=\{X\subseteq U:S^{-}(X)=X\}
\end{equation*}%
be a collection of subsets of $U$. Then $K$ is called a topology over $U$.
\end{theorem}

\begin{proof}

\begin{description}
\item[$C_{1})$] If $X=\emptyset $, then by Theorem 14, $S^{-}(\emptyset
)=\emptyset $. Hence $\emptyset \in K$. If $X=U$, then by Theorem 14, $%
S^{-}(U)=U$. Hence $U\in \tau $.

\item[$C_{2})$] Let for each $i\in I,\ A_{i}\in K$, i.e., $S^{-}(A_{i})=A_{i}
$. Then there exists a $j\in I,\ A_{j}\in K$ such that $\dbigcap\limits_{i%
\in I}A_{i}\subseteq A_{j}$. From Theorem 21, $S^{-}(\dbigcap\limits_{i\in
I}A_{i})\subseteq S^{-}(A_{j})$. Since $A_{j}\in K,\ S^{-}(A_{j})=A_{j}$.
Hence $S^{-}(\dbigcap\limits_{i\in I}A_{i})\subseteq A_{j}$. Since this
property is satisfied for each $j\in I$, we get%
\begin{equation}
S^{-}(\dbigcap\limits_{i\in I}A_{i})\subseteq \dbigcap\limits_{i\in I}A_{i}
\label{3}
\end{equation}%
Also by Theorem 14, we know that%
\begin{equation}
\dbigcap\limits_{i\in I}A_{i}\subseteq S^{-}(\dbigcap\limits_{i\in I}A_{i})
\label{4}
\end{equation}%
From (\ref{3}) and (\ref{4}), we obtain that $S^{-}(\dbigcap\limits_{i\in
I}A_{i})=\dbigcap\limits_{i\in I}A_{i}$. Therefore $\dbigcap\limits_{i\in
I}A_{i}\in K$. \ \ \ \ \ \ \ \ \ \ \ \ \ \ \ \ \ \ \ \ \ \ \ \ \ \ \ \ \ \ \
\ \ \ \ \ \ \ \ \ \ \ \ \ \ \ \ \ \ \ \ \ \ \ \ \ \ \ \ \ \ \ \ \ \ \ \ \ \
\ \ \ \ \ \ \ \ \ \ \ \ \ \ \ \ \ \ \ \ \ \ \ \ \ \ \ \ \ \ \ \ \ \ \ \ \ \
\ 

\item[$C_{3})$] Let $A,B\in K$. Hence $S^{-}(A)=A$ and $S^{-}(B)=B$. By
Corollary 23, we obtain that $S^{-}(A\cup B)=S^{-}(A)\cup S^{-}(B)=A\cup B$.
Hence $S^{-}(A\cup B)=A\cup B$. Therefore $A\cup B\in K$.
\end{description}
\end{proof}

\begin{remark}
Let $U$ be a nonempty universe set and $S=(U,C_{G})$ be a soft covering
approximation space. We can set up a topology over $U$ when we consider the
soft covering of the universe as a subbase.
\end{remark}

\begin{example}
Let $S=(U,C_{G})$ be a soft covering approximation space where $%
U=\{h_{1},h_{2},h_{3},h_{4},h_{5}\},\ E=\{e_{1},e_{2},e_{3}\},\
F(e_{1})=\{h_{1},h_{2},h_{3}\},\ F(e_{2})=\{h_{3},h_{4}\},\
F(e_{3})=\{h_{4},h_{5}\}$. Then
\end{example}

\begin{align*}
\mathit{S} & =C_{G}=\left\{
\{h_{1},h_{2},h_{3}\},\{h_{3},h_{4}\},\{h_{4},h_{5}\}\right\} \\
& \dbigcap \\
\beta & =\left\{
\{h_{1},h_{2},h_{3}\},\{h_{3},h_{4}\},\{h_{4},h_{5}\},\{h_{3}\},\{h_{4}\}%
\right\} \\
& \dbigcap \\
\tau &
=\{\emptyset,U,\{h_{1},h_{2},h_{3},h_{4}\},\{h_{3},h_{4},h_{5}\},%
\{h_{1},h_{2},h_{3}\},\{h_{3},h_{4}\},\{h_{4},h_{5}\},\{h_{3}\},\{h_{4}\}\}
\end{align*}

\section{Relationship between concepts of topology and soft covering lower
and upper approximations}

In soft covering based rough set theory the reference space is the soft
covering approximation space. We will consider the soft covering of the
universe as a subbase for topology and we will obtain the closure, the
interior and the boundary of a set with respect to this topology, then we
will compare these concepts with the soft covering upper approximation, the
soft covering lower approximation and the soft covering boundary region of a
set.

\begin{proposition}
Let $S=(U,C_{G})$ be a soft covering approximation space and $X\subseteq U$.
The soft covering lower approximation is contained in the interior of a set
defined by taking this soft covering as a subbase for topology.
\end{proposition}

\begin{proof}
Let $C_{G}$ be a soft covering of the universe $U$, $X\subseteq U$ and $x\in
S_{-}(X)$. Then, $\exists F(e)\in S_{-}(X)$ such that $x\in F(e)$. Since $%
F(e)$ is an element of subbase for the topology defined on $U$ then every $%
F(e)\in C_{G}$ is open hence $x\in\cup\left\{ F(e)\subseteq U:F(e)\subseteq
X\ open\right\} $. Thus $x\in int(X)$ and $S_{-}(X)\subseteq int(X)$.
\end{proof}

\begin{proposition}
Let $S=(U,C_{G})$ be a soft covering approximation space and $X\subseteq U$.
The soft covering upper approximation of $X$ can not be compared with the
closure of $X$ with respect to the topology induced by soft covering.
\end{proposition}

\begin{corollary}
Let $S=(U,C_{G})$ be a soft covering approximation space and $X\subseteq U$.
The soft covering boundary region of $X$ can not be compared with the
boundary of $X$ with respect to the topology induced by soft covering.
\end{corollary}

\begin{example}
Let $S=(U,C_{G})$ be a soft covering approximation space, where $%
U=\{h_{1},h_{2},h_{3},h_{4},h_{5}\},\ E=\{e_{1},e_{2},e_{3},\},\
F(e_{1})=\{h_{1},h_{2},h_{3}\},\ F(e_{2})=\{h_{3},h_{4}\},\
F(e_{3})=\{h_{4},h_{5}\}$. Suppose that $X=\{h_{2},h_{3},h_{4}\}$, then $%
S_{-}(X)=\{h_{3},h_{4}\},~S^{-}(X)=\{h_{1},h_{2},h_{3},h_{4}\},\
BND_{S}(X)=\{h_{1},h_{2}\}$ and by using the Example 27, we get $%
int(X)=\{h_{3},h_{4}\},\ cl(X)=U,~Bnd(X)=\{h_{1},h_{2},h_{5}\}$. Thus we
obtain, $S_{-}(X)\subseteq int(X),~S^{-}(X)\subseteq cl(X)$ and $%
BND_{S}(X)\subseteq Bnd(X)$.\newline
Also, suppose that $Y=\{h_{1},h_{4},h_{5}\}$, then $S_{-}(Y)=\{h_{4},h_{5}%
\},~S^{-}(Y)=U,~BND_{S}(Y)=\{h_{1},h_{2},h_{3}\}$ and by using the Example
27, we get $int(Y)=\{h_{4},h_{5}\},~cl(Y)=\{h_{1},h_{2},h_{4},h_{5}%
\},~Bnd(Y)=\{h_{1},h_{2}\}$. Thus we obtain, $S_{-}(Y)\subseteq
int(Y),~cl(Y)\subseteq S^{-}(Y)$ and $Bnd(Y)\subseteq BND_{S}(Y)$.
\end{example}

\section{Special condition of soft covering approximation space}

\begin{definition}
\cite{b} A soft set $G=(F,E)$ over $U$ is called a partition soft set if $%
\{F(e):e\in E\}$ forms a partition of $U$.
\end{definition}

\begin{theorem}
\cite{b} Let $G=(F,E)$ be a partition soft set over $U$ and $P=(U,G)$ be a
soft covering approximation space. Define an equivalence relation $R$ on $U$
by%
\begin{equation*}
(x,y)\in R\Longleftrightarrow\exists e\in E,\ \{x,y\}\subseteq F(e)
\end{equation*}
for all $x,y\in U$. Then, for all $X\subseteq U$,%
\begin{equation*}
R_{-}(X)=P_{-}(X)\ and\ R^{-}(X)=P^{-}(X).
\end{equation*}
\end{theorem}

\begin{theorem}
Let $S=(U,C_{G})$ be a soft covering approximation space and $X\subseteq U$.
If $G=(F,E)$ is a partition soft set then the soft covering upper
approximation and the soft covering lower approximation of $X$ are equal to
the closure and the interior of the set with respect to the topology induced
by this covering, respectively.
\end{theorem}

\begin{proof}
Let $G=(F,E)$ be a partition soft set then $R_{-}(X)=S_{-}(X)$ and $%
R^{-}(X)=S^{-}(X)$. And we know that $R_{-}(X)=int(X)$ and $R^{-}(X)=cl(X)$.
Hence we conclude that $S_{-}(X)=int(X)$ and $S^{-}(X)=cl(X)$.
\end{proof}

\begin{example}
Let $G=(F,E)$ be a partition soft set and $S=(U,C_{G})$ be a soft covering
approximation space, where $U=\{h_{1},h_{2},h_{3},h_{4},h_{5},h_{6},h_{7}\},%
\ E=\{e_{1},e_{2},e_{3}\},\ F(e_{1})=\{h_{1},h_{2}\},\
F(e_{2})=\{h_{3},h_{4}\}$ and $F(e_{3})=\{h_{5},h_{6},h_{7}\}$. Then we get%
\begin{align*}
\mathit{S} & =C_{G}=\{\{h_{1},h_{2}\},\{h_{3},h_{4}\},\{h_{5},h_{6},h_{7}\}\}
\\
& \dbigcap \\
\beta & =\{\{h_{1},h_{2}\},\{h_{3},h_{4}\},\{h_{5},h_{6},h_{7}\}\} \\
& \dbigcap \\
\tau & =\{\{h_{1},h_{2}\},\{h_{3},h_{4}\},\{h_{5},h_{6},h_{7}\}\}
\end{align*}
Suppose that $X=\{h_{1},h_{2},h_{3}\}$, then $S_{-}(X)=\{h_{1},h_{2}\},\
S^{-}(X)=\{h_{1},h_{2},h_{3},h_{4}\},\ BND_{S}(X)=\{h_{3},h_{4}\}$ and $%
int(X)=\{h_{1},h_{2}\},\ cl(X)=\{h_{1},h_{2},h_{3},h_{4}\},\
Bnd(X)=\{h_{3},h_{4}\}$. Hence $S_{-}(X)=int(X),\ S^{-}(X)=cl(X)$ and $%
BND_{S}(X)=Bnd(X)$. \qquad\ \ \ \ \ 
\end{example}

\end{document}